\documentclass{article}
\usepackage{amsmath,amssymb,amsthm}
\normalsize\textwidth140mm\hoffset=-10mm%
\normalsize\parindent=1em%
\theoremstyle{plain} \newtheorem*{theorem}{Theorem}
\theoremstyle{plain} \newtheorem{lemma}{Lemma}[section]
\theoremstyle{remark} \newtheorem{remark}[lemma]{Remark}
\theoremstyle{plain} \newtheorem{proposition}[lemma]{Proposition}
\theoremstyle{plain} \newtheorem{corollary}[lemma]{Corollary}
\title{Jordan Matsuo algebras over fields of characteristic 3}
\author{Takahiro Yabe}
\date{August 7, 2017}
\begin{document}
\large
\baselineskip17pt
\maketitle

\begin{abstract}
\noindent 
The Matsuo algebra associated with a connected Fischer space is shown to be a Jordan algebra over a field of characteristic $3$ if and only if the Fischer space is isomorphic to either the affine space of order $3$ or the Fischer space associated with the symmetric group. 
The proof uses a characterization of the affine spaces of order $3$ and equivalence of Jordan and linearized Jordan identities over a field of characteristic $3$ in case the algebra is spanned by idempotents. 
\end{abstract}

\section{Introduction}
From any partial triple system $(\mathcal{P},\mathcal{L})$, a family of commutative nonassociative algebras  parametrized by one parameter $\delta$ was constructed in \cite{mat03} (cf. \cite{mat05}).
Such algebras over a field $\mathbb{F}$ are called the \textit{Matsuo algebras} associated with $(\mathcal{P},\mathcal{L})$ and denoted by $M((\mathcal{P},\mathcal{L}),\delta,\mathbb{F})$, where $\delta\in\mathbb{F}$.
If $(\mathcal{P},\mathcal{L})$ belongs to Fischer spaces, the geometric counterparts to 3-transposition groups, and if $2\delta\neq0,1$, then the associated Matsuo algebra becomes an axial algebra of Jordan type $2\delta$ (\cite{hrs15}).

Recently, T.\ De Medts and F.\ Rehren considered the problem  of classifying Jordan Matsuo algebras associated with Fischer spaces in \cite{dr15}.
Note that the parameter $\delta$ must be $\frac{1}{4}$ if the associated Matsuo algebra is Jordan.

The purpose of this paper is to investigate this classification over fields of characteristic $3$.
Our main result is the following theorem.
\begin{theorem} 
\sl
Let $(\mathcal{P},\mathcal{L})$ be a connected Fischer space of rank $n$ and $\mathbb{F}$ a field of characteristic $3$.
Then the associated Matsuo algebra $M((\mathcal{P},\mathcal{L}), \frac{1}{4},\mathbb{F})$ is Jordan if and only if $(\mathcal{P},\mathcal{L})$ is isomorphic to either the Fischer space associated with the symmetric group $(\mathit{Sym}(n+1),(1,2)^{\mathit{Sym}(n+1)})$ or the $(n-1)$-dimensional affine space $AG(n-1,3)$ of order three.
\end{theorem}
By combining this result and the argument of \cite{dr15}, the classification of finite-dimensional Jordan Matsuo algebras associated with Fischer spaces over all fields of characteristic not $2$ is completed.
 (The second case in Theorem had been erroneously omitted in \cite{dr15} 
\footnote{The factor $\frac{1}{32}$ mentioned in page 339 of \cite{dr15} is in fact $\frac{3}{32}$. As it is $0$ in a field of characteristic $3$, one cannot conclude that the algebra considered there is not Jordan over a field of characteristic $3$, and it is indeed Jordan as we will see in the present paper.}.)

Section 2 recalls the definitions of 3-transposition groups, Fischer spaces and Matsuo algebras.
Section 3 gives a characterization of affine spaces of order three by means of their automorphisms.
This characterization is used in section 5.
Section 4 shows that the Matsuo algebras associated with the affine spaces of order 3 is Jordan over a field of characteristic $3$.
In Section 5, the proof of Theorem is completed.

\bigbreak
The author is grateful to Prof.\ Atsushi Matsuo for guidance throughout the work and to Prof.\ Dr.\ Tom De Medts for useful comments. 
The author thanks Prof.\ Hiroshi Yamauchi for careful reading of the manuscript of the paper.

\section{3-transposition groups, Fischer spaces and Matsuo algebras} 
A \textit{3-transposition group} is a pair $(G,D)$ where $G$ is a group and $D$ is a normal subset of $G$ satisfying $|g|=2$ and $|gh|\leq 3$ for all $g, h\in D$. 
The \textit{rank} of $(G,D)$ is the minimum size of subsets of $D$ generating $(G,D)$ and denoted by $r(G,D)$.
If $r(G,D)<\infty$, then $G$ is finite.

Let $S$ be a subset of a group $G$.
The \textit{noncommuting graph} of $S$ is a graph $\mathcal{S}$ with vertices $S$ and edges $\{g,h\}$ for $g,h\in S$ with $gh\neq hg$.
A 3-transposition group $(G,D)$ is said to be \textit{connected} if the noncommuting graph of $D$ is connected.
In this paper, for a group $G$, $g\in G$ and $S\subset G$, $g^G$ and $S^G$ denote $\{g^h:=h^{-1}gh\mid h\in G\}$ and $\bigcup_{g\in S}g^G$ respectively.

For the symmetric group $\mathit{Sym}(n)$, the pair $(\mathit{Sym}(n),(1,2)^{\mathit{Sym}(n)})$ is a 3-transposition group of rank $n-1$.
Set $\tilde{G}=\langle\tilde{g}_1,\tilde{g}_2,\tilde{g}_3,\tilde{g}_4\mid\tilde{g}_i^2=1, 1\leq i\leq4\rangle$ and $\tilde{D}=\{\tilde{g}_1,\tilde{g}_2,\tilde{g}_3,\tilde{g}_4\}^{\tilde{G}}$. 
Then the pair $(G,D)$ with $G=\tilde{G}/\langle g^h(h^g)^{-1}\mid g,h\in \tilde{D}\rangle$ and $D=\{g_1,g_2,g_3,g_4\}^{G}$ is a 3-transposition group of rank $4$, where $g_i$ is the image of $\tilde{g}_i$. 
This group $G$ is called \textit{M.\ Hall's group} $3^{10}:2$ (this group was defined in \cite{mhal62}).

A \textit{partial triple system} is a pair $(\mathcal{P},\mathcal{L})$ where $\mathcal{P}$ is a set and $\mathcal{L}$ is a set of $3$-subsets of $\mathcal{P}$ such that $|l_1\cap l_2|\leq1$ for all distinct $l_1,l_2\in\mathcal{L}$. 
We call an element of $\mathcal{P}$ a \textit{point} and an element of $\mathcal{L}$ a \textit{line}.
A \textit{subsystem} of $(\mathcal{P},\mathcal{L})$ is a partial triple system $(\mathcal{P}^{\prime},\mathcal{L}^{\prime})$ with $\mathcal{P}^{\prime}\subset\mathcal{P}$ and $\mathcal{L}\supset\mathcal{L}^{\prime}\supset\{l\in\mathcal{L}\mid|l\cap\mathcal{P}^{\prime}|\ge2\}$.
A subsystem of $(\mathcal{P},\mathcal{L})$ is said to be \textit{generated} by a subset $S$ of $\mathcal{P}$ if it is the minimum subsystem whose point set contains $S$.
A \textit{plane} of a partial triple system $(\mathcal{P},\mathcal{L})$ is a subsystem of $(\mathcal{P},\mathcal{L})$ generated by $l_1\cup l_2$ for a pair $(l_1,l_2)$ of distinct and intersecting lines.
The \textit{rank} of $(\mathcal{P},\mathcal{L})$ is the minimum size of subsets of $\mathcal{P}$ generating $(\mathcal{P},\mathcal{L})$ itself.
The rank of $(\mathcal{P},\mathcal{L})$ is denoted by $r(\mathcal{P},\mathcal{L})$.
Two partial triple systems are said to be \textit{isomorphic} if a bijection between the point sets induces a bijection between the line sets.

Two distinct points $p$ and $q$ are said to be \textit{collinear} if both $p$ and $q$ are contained in a common line. 
Let $\sim$ be the equivalence relation on $\mathcal{P}$ generated by $p\sim q$ for two collinear points $p$ and $q$.
For a $\sim$-equivalence class $\mathcal{P}^{\prime}$ of $\mathcal{P}$, we call the subsystem generated by $\mathcal{P}^{\prime}$ a \textit{connected component} of $(\mathcal{P},\mathcal{L})$.
The point set of a connected component of $(\mathcal{P},\mathcal{L})$ is a $\sim$-equivalence class of $\mathcal{P}$.
$(\mathcal{P},\mathcal{L})$ is said to be \textit{connected} if it has a unique connected component.

We recall the two important examples of partial triple systems.
The \textit{dual affine plane} $DA(2,2)$ \textit{of order two} is the partial triple system obtained by discarding a point and all lines including the point from the projective plane of order two.
Thus for the point set $\{p,q,r,s,t,u\}$, we may take the line set to be $\{\{p,r,s\},\{p,t,u\},\{q,r,t\},\{q,s,u\}\}$.
The second example is the $n$-dimensional \textit{affine space} $AG(n,3)$ \textit{of order three}.
The points are the vector of $\mathbb{F}_3^n$ and the lines are the $1$-flats of $\mathbb{F}_3^n$.
Thus for the point set $\mathbb{F}_3^n$, we may take the line set to be $\{\{x,y,-x-y\}\mid x,y\in\mathbb{F}_3^n, x\neq y\}$.

A \textit{Fischer space} is a partial triple system whose planes are isomorphic to either $DA(2,2)$ or $AG(2,3)$.
A Fischer space is said to be \textit{of affine type} if its planes are isomorphic to $AG(2,3)$. 

Let $(G,D)$ be a 3-transposition group.
The pair $(\mathcal{P},\mathcal{L})$ with $\mathcal{P}=D$ and $\mathcal{L}=\{\{c,d,c^d\}\mid c,d\in D, |cd|=3\}$ is a Fischer space.
We call this Fischer space the Fischer space associated with $(G,D)$ and denote it by $FS(G,D)$.
Let $FSS_n$ denote $FS(\mathit{Sym}(n),(1,2)^{\mathit{Sym}(n)})$.
Note that $r(FS(G,D))=r(G,D)$ and $FS(G,D)$ is connected if and only if $(G,D)$ is connected.

Let $(\mathcal{P},\mathcal{L})$ be a Fischer space.
For each $p\in\mathcal{P}$, the \textit{associated transposition} $\tau(p)$ is the automorphism of $(\mathcal{P},\mathcal{L})$ which fixes the points not collinear with $p$ and switches $q$ and $r$ if $\{p,q,r\}\in\mathcal{L}$.
Note that the mapping $\tau$ from $\mathcal{P}$ to the automorphism group of $(\mathcal{P},\mathcal{L})$ satisfies the following properties:
\begin{itemize}
\item $\tau(p)^2=id$ for all $p\in\mathcal{P}$.
\item $\tau(p)\tau(q)=\tau(q)\tau(p)$ if $p$ and $q$ are not collinear.
\item $\tau(p)^{\tau(q)}=\tau(q)^{\tau(p)}$ if $p$ and $q$ are collinear.
\item $\tau(p^{\tau(q)})=\tau(p)^{\tau(q)}$ for all $p,q\in\mathcal{P}$, where $p^f=f(p)$ for $p\in\mathcal{P}$ and an automorphism $f$ of $(\mathcal{P},\mathcal{L})$.
\end{itemize}
The pair $(G,D)$ with $G=\langle\tau(p)\mid p\in\mathcal{P}\rangle$ and $D=\{\tau(p)\mid p\in\mathcal{P}\}$ is a 3-transposition group.
If $(\mathcal{P},\mathcal{L})$ is connected, then the associated Fischer space $FS(G,D)$ is isomorphic to $(\mathcal{P},\mathcal{L})$.
So, for a Fischer space $(\mathcal{P},\mathcal{L})$, there exists a 3-transposition group $(G,D)$ such that $(\mathcal{P},\mathcal{L})$ is isomorphic to $FS(G,D)$.
For example, $AG(n,3)$ is isomorphic to $FS(\mathbb{F}_3^n\rtimes\mathbb{F}_3^{\times},\{(v,-1)\mid v\in\mathbb{F}_3^n\})$, where the action of $\mathbb{F}_3^{\times}$ is the scalar multiplication of $\mathbb{F}_3^n$.

Let $(\mathcal{P},\mathcal{L})$ be a partial triple system and $\mathbb{F}$ a field with $\mathrm{ch}(\mathbb{F})\neq 2$.
Choose a constant $\delta\in \mathbb{F}$.
Then, the \textit{Matsuo algebra} $M((\mathcal{P},\mathcal{L}),\delta,\mathbb{F})$ associated with $(\mathcal{P},\mathcal{L})$ is the $\mathbb{F}$-space $\bigoplus _{p\in\mathcal{P}}\mathbb{F}a(p)$ with the multiplication given as follows:
\begin{enumerate}
\item[\rm(i)] $a(p)a(p)=a(p)$,
\item[\rm(ii)] $a(p)a(q)=0$ if $p$ and $q$ are not collinear,
\item[\rm(iii)] $a(p)a(q)=\delta (a(p)+a(q)-a(r))$ if $p$ and $q$ are collinear and $\{p,q,r\}$ is the common line.
\end{enumerate}
In \cite{hrs15}, it is proved that $M((\mathcal{P},\mathcal{L}), \delta , \mathbb{F})$ is a primitive axial algebra of Jordan type $2\delta$ if $(\mathcal{P},\mathcal{L})$ is a Fischer space and $\delta\neq0, \frac{1}{2}$.

Let $\{(\mathcal{P}_i,\mathcal{L}_i)\}_{i\in I}$ be the set of connected components of a Fischer space $(\mathcal{P},\mathcal{L})$.
Then  
$$M((\mathcal{P},\mathcal{L}),\frac{1}{4},\mathbb{F})\cong\bigoplus_{i\in I}M((\mathcal{P}_i,\mathcal{L}_i),\frac{1}{4},\mathbb{F}).$$ 
From now on, we assume that Fischer spaces are connected. 

\section{The structure of affine spaces as Fischer spaces}
In this section, we give a characterization of affine spaces using the mapping $\tau$ defined for each Fischer space in the preceding section.

First, let us give an alternative construction of $AG(n-1,3)$.
Let $I_n=\{1,2,\ldots,n\}$, $R_n^{(k)}$ be the set of the sequences $(i_{-1},i_1,\ldots,i_k)$ of the elements of $I_n$ and 
$R_n=\bigcup_{k=0}^{\infty}R_n^{(k)}$.
If a sequence $x\in R_n$ is in $R_n^{(k)}$, we say that the length of $x$ is $k$.
Let $\rho :R_n\to \mathrm{Map}(R_n,R_n)$ be the mapping defined by setting 
$$\rho(j_{-1},j_1,\ldots,j_k)(i_{-1},i_1,\ldots,i_m)=(i_{-1},i_1,\ldots,i_m,j_k,\ldots,j_1,j_{-1},j_1,\ldots,j_k)$$
for $(j_{-1},j_1,\ldots,j_k),(i_{-1},i_1,\ldots,i_m)\in R_n$.

Let $\approx$ be the equivalence relation on $R_n$ generated by the following elementary equivalences:
\begin{itemize}
\item $\rho(x_k)\cdots\rho(x_1)(x_1)\approx\rho(x_k)\cdots\rho(x_2)(x_1)$ for $k\geq1$ and $x_1,\ldots,x_k\in R_n$.
\item $\rho(x_k)\cdots\rho(x_2)\rho(x_2)(x_1)\approx\rho(x_k)\cdots\rho(x_3)(x_1)$ for $k\geq2$ and $x_1,\ldots,x_k\in R_n$.
\item $\rho(x_k)\cdots\rho(x_2)(x_1)\approx\rho(x_k)\cdots\rho(x_3)\rho(x_1)(x_2)$ for $k\geq2$ and $x_1,\ldots,x_k\in R_n$.
\item $\rho(x_k)\cdots\rho(x_2)(x_1)\approx\rho(x_k)\cdots\rho(x_5)\rho(x_2)\rho(x_3)\rho(x_4)(x_1)$ for $k\geq4$ and $x_1,\ldots,x_k\in R_n$.
\end{itemize}
Let $\mathcal{Q}_n=R_n/\approx$, $f$ the natural surjection from $R_n$ to $\mathcal{Q}_n$ and  $q_i=f(i)$ for $i\in I_n$.
Then the mapping $\rho$ induces a mapping $\sigma:\mathcal{Q}_n\to\mathrm{Map}(\mathcal{Q}_n,\mathcal{Q}_n)$ because $\rho(y)(x)\approx\rho(x)(y)\approx\rho(x)(z)\approx\rho(z)(x)$ if $y\approx z$. 
Set $\mathcal{M}_n=\{\{p,q,p^{\sigma(q)}\}\mid p,q\in\mathcal{Q}_n, p\neq q\}$, where $p^{\sigma(q)}=\sigma(q)(p)$.
Then, the pair $(\mathcal{Q}_n,\mathcal{M}_n)$ is a partial triple system.

\begin{proposition} 
\sl
$(\mathcal{Q}_n,\mathcal{M}_n)$ is isomorphic to $AG(n-1,3)$.
\label{lem}
\end{proposition}
In order to show this proposition, let us consider the set $S_n$ of sequences $(i_{-1},i_1,\ldots,i_m)\in R_n$ satisfying the following conditions:
\begin{itemize}
\item $|\{j\in\{-1,1,\ldots,m\}\mid i_j=i\}|\leq2$ for all $i\in I_n$,
\item $\forall j\in\{-1,1,\ldots,m-2\}, i_j\leq i_{j+2}$,
\item $\forall k\in I_n$, $|\{j\in\{-1,1,\ldots,m\}\mid i_j=k\}|\leq 2$,
\item $i_j=i_k \Rightarrow j-k=\pm 2$, $j\in  2\mathbb{Z}$ and $i_j>i_l$ for all $l\in 2\mathbb{Z}+1$.
\end{itemize}
\begin{lemma}
\sl
The restriction $f|_{S_n}:S_n\to\mathcal{Q}_n$ is surjective.
\end{lemma}
\begin{proof}
First note the following invariance of the surjection $f:R_n\to\mathcal{Q}_n$:
\begin{enumerate}
\item[(1)]$f(i_{-1},\ldots,i_j,i_{j+1},i_{j+2},\ldots,i_m)=f(i_{-1},\ldots,i_{j+2},i_{j+1},i_j,\ldots,i_m)$ for all $j\in\{-1,1,\ldots,m-2\}$.
\item[(2)]$f(i_{-1},i_1=i_{-1},i_2,\ldots,i_m)=f(i_{-1},i_2,\ldots,i_m)$.
\item[(3)]$f(i_{-1},\ldots,i_j,i_{j+1}=i_j,\ldots,i_m)=f(i_{-1},\ldots,i_{j-1},i_{j+2},\ldots,i_m),$ for all $j\in\{1,2,\ldots,m-1\}$.
\item[(4)]$f(i_{-1},\ldots,i_j,i_{j+1},i_{j+2}=i_j,\ldots,i_m)=f(i_{-1},\ldots,i_{j+1},i_j,i_{j+1},\ldots,i_m)$ for all $j\in\{1,2,\ldots,m-2\}$.
\end{enumerate}
Let $x\in\mathcal{Q}_n$ and $(i_{-1}, i_1,\ldots,i_m)\in R_n$ be a sequence satisfying $f(i_{-1}, i_1,\ldots,i_m)=x$.
Firstly, if $i_j=i_k=i_l=s$ and $j$, $k$ and $l$ are distinct, then by the invariance above, there exists a shorter sequence $y$ such that $f(y)=x$. 
Hence we may assume $|\{j\mid i_j=s\}|\leq2$ for all $s\in I_n$. 
Secondly, suppose $i_j=i_k=s$ and $j$ or $k$ is odd. Then, we can reduce $|\{r\mid i_r=s\}|$ by (1), (2), and (3).
So, we may assume $i_j=i_k, j\neq k \Rightarrow j,k\in  2\mathbb{Z}$.
Thirdly, by (1) and (4), if $i_j=i_k=s$, $i_l=t$ and $l$ is odd, then the value of $f$ does not change if we make $i_j=i_k=t$ and $i_l=s$. 
So we may assume $i_j>i_l$ for all $l\in 2\mathbb{Z}+1$ if there exists $k$ such that $i_j=i_k$ and $j\neq k$.
Finally, by (1), we may assume $\forall j\in\{-1,1,\ldots,m-2\}, i_j\leq i_{j+2}$. 
Therefore, we may assume the sequence  $(i_{-1}, i_1,\ldots,i_m)$ is an element of $S_n$.
So, for all $x\in\mathcal{Q}_n$, there exist an element $y$ of $S_n$ such that $x=f(y)$.
Therefore $f|_{S_n}$ is surjective.
\end{proof}
\begin{lemma}
\sl
$|S_n|=3^{n-1}$. 
\end{lemma}
\begin{proof}
Set $T_k=\{(i_{-1},i_1,\ldots,i_m)\in S_n\mid\{i_{-1},i_1,\ldots,i_m\}=\{1,\ldots,k\}\}$ for $1\leq k\leq n$.
For $x=(i_{-1}, i_1,\ldots ,i_m)\in S_n$, set $s(x)=|\{j\mid i_j=i_{j+2}\}|$ and $t(x)=|\{j\in 2\mathbb{Z}\mid i_j\neq i_{j\pm 2}\}|$. 
Then, for $x\in T_k$, $2s(x)+t(x)=m^{\prime}=k-s(x)-t(x)-1$ if $m=2m^{\prime}$ and $2s(x)+t(x)=m^{\prime}=k-s(x)-t(x)-2$ if $m=2m^{\prime}+1$. 
Since $x\in T_k$ is determined by $\{i_{2j}\mid i_{2j}\neq i_{2j\pm 2}\}$, 
\begin{eqnarray*}
|T_k|
&=&\sum_{0\leq t\leq \frac{k-1}{2}, k-2t\notin 3\mathbb{Z}}\binom{k}{t}\\
&=&\frac{1}{2}\sum_{0\leq t\leq k, k+t\notin 3\mathbb{Z}}\binom{k}{t}\\
&=&\frac{1}{2}(\sum_{t=0}^k\binom{k}{t}-\sum_{0\leq t\leq k, k+t\in 3\mathbb{Z}}\binom{k}{t})\\
&=&\frac{1}{2}(2^k-\frac{1}{3}(2^k+\omega^k(1+\omega)^k+\omega^{2k}(1+\omega^2)^k))\\
&=&\frac{1}{3}(2^k-(-1)^k),
\end{eqnarray*}
where $\omega=\frac{-1-\sqrt{-3}}{2}$. 
Therefore, 
\begin{eqnarray*}
|S_n|=\sum_{k=1}^n\binom{n}{k}|T_k|=3^{n-1}.
\end{eqnarray*}
\end{proof}
\begin{proof}[Proof of Proposition \ref{lem}]
Let $\{e_1,\ldots,e_{n-1}\}$ be a basis of $\mathbb{F}_3^{n-1}$ and $e_n=0\in\mathbb{F}_3^{n-1}$.
By the structure of $AG(n,3)$, the mapping from $\mathcal{Q}_n$ to $\mathbb{F}_3^{n-1}$ such that $q_{i_1}^{\sigma(q_{i_2})\cdots\sigma(q_{i_k})}\mapsto e_{i_1}^{\tau(e_{i_2})\cdots\tau(e_{i_k})}$ is well-defined and surjective.
By the lemmas above, $|\mathcal{Q}_n|\leq|S_n|=3^{n-1}$ and hence this map is bijective.
It is easy to see that this mapping induces a surjection from $\mathcal{M}_n$ to the set of lines of $AG(n-1,3)$.
So, this map induces an isomorphism of Fischer spaces.
\end{proof}

The following corollary gives the characterization of affine spaces mentioned at the beginning of this section.
\begin{corollary} 
\sl
Let $(\mathcal{P},\mathcal{L})$ be a connected Fischer space of affine type. 
If $(\mathcal{P},\mathcal{L})$ is of rank $n$ and $\tau(x)\tau(y)\tau(z)=\tau(z)\tau(y)\tau(x)$ for all $x,y,z\in\mathcal{P}$, then $(\mathcal{P},\mathcal{L})$ is isomorphic to $AG(n-1,3)$.
\label{cor}
\end{corollary}
\begin{proof}
Let $\{p_1,\ldots,p_n\}$ be a subset of $\mathcal{P}$ generating $(\mathcal{P},\mathcal{L})$.
Then, by the structure of $(\mathcal{P},\mathcal{L})$, there exists a surjection $\pi:\mathbb{F}_3^{n-1}\to\mathcal{P}$ such that $e_i\mapsto p_i$ and it induces the quotient map from $AG(n-1.3)\cong(\mathcal{Q}_n,\mathcal{M}_n)$ to $(\mathcal{P},\mathcal{L})$.
It suffices to show that $\pi$ is injective.
Let $v,w\in\mathbb{F}_3^{n-1}$ and $\pi(v)=\pi(w)$. 
Then $\pi(u)=\pi(u+v-w)$ for all $u\in\mathbb{F}_3^{n-1}$.
Let $S$ be a subset of $\mathbb{F}_3^{n-1}$ such that the image of $S$ is a basis of $\mathbb{F}_3^{n-1}/\mathbb{F}_3(v-w)$.  
Then, $(\mathcal{P},\mathcal{L})$ is generated by $\{\pi(u)\mid u\in S\}\cup\{\pi(0)\}$.
Since $(\mathcal{P},\mathcal{L})$ is of rank $n$, $v=w$.
Therefore $\pi$ is injective.
\end{proof}

\section{Matsuo algebras associated with affine spaces}
Recall that a Jordan algebra is a commutative nonassociative algebra $A$ satisfying $(a^2b)a=a^2(ba)$ for all $a,b\in A$. 
It is known that the linearized Jordan identity
$$((xz)y)w+((zw)y)x+((wx)y)z-(xz)(yw)-(zw)(yx)-(wx)(yz)=0$$
holds for all $x,y,z,w\in A$ if $A$ is a Jordan algebra over a field of characteristic not $2$. (For example, see \cite{mcc04}, Proposition 1.8.5 (1).)
Conversely, a commutative nonassociative algebra $A$ over a field of characteristic not $3$ is a Jordan algebra if the linearized Jordan identity holds for all $x,y,z,w\in A$.
From now on, let $J(x,y,z,w)$ denote the left-hand side of the linearized Jordan identity.

\begin{lemma} 
\sl
Let $\mathbb{F}$ be a field of charateristic 3 and $A$ a commutative nonassociative algebra over $\mathbb{F}$ spanned by idempotents. Then $A$ is Jordan if the linearized Jordan identity holds for all $x,y,z,w\in A$.
\end{lemma}
\begin{proof} 
Suppose $a=\sum_{i=1}^nx_ia_i\in A$, where $x_1,\ldots,x_n\in\mathbb{F}$ and $a_1,\ldots,a_n\in A$ are idempotents.
Then, for all $b\in A$,
\begin{eqnarray*} 
(a^2b)a-a^2(ba)&=&\sum_{1\leq i\leq n}x_i^3((a_i^2b)a_i-a_i^2(ba_i))\\
&&+\sum_{1\leq i,j\leq n.i\neq j}x_i^2x_j((a_i^2b)a_j+2((a_ia_j)b)a_i-a_i^2(ba_j)-2(a_ia_j)(ba_i))\\
&&+2\sum_{1\leq i<j<k\leq n}x_ix_jx_kJ(a_i,a_j,b,a_k)\\
&=&\sum x_i^3((a_ib)a_i-a_i(ba_i))+\sum x_i^2x_jJ(a_i,a_i,b,a_j)\\
&=&0.
\end{eqnarray*} 
Since $A$ is spanned by idempotents, the Jordan identity holds for all $a,b\in A$.
\end{proof}

Let $\mathbb{F}$ be a field of characteristic $3$.
Let us show that $M(AG(n,3),\frac{1}{4},\mathbb{F})$ is Jordan.
To this end, it suffices to verify the linearized Jordan identity because Matsuo algebras are spanned by idempotents.
Since this identity is linear, it suffices to verify $J(a(x),a(y),a(z),a(w))=0$ for $x,y,z,w\in\mathbb{F}_3^n$. 
Since $a(x)a(y)=a(x)+a(y)-a(-x-y)$ for all $x,y\in\mathbb{F}_3^n$,
\begin{eqnarray*}
J(a(x),a(y),a(z),a(w))=3a(x+y+z+w)-3a(y-x-z-w)=0.
\end{eqnarray*}
\begin{remark}
If $\mathrm{ch}(\mathbb{F})\neq3$ and $\{x,y,z,w\}$ generates $AG(3,3)$, then 
$$J(a(x),a(y),a(z),a(w))=\frac{3}{64}(a(x+y+z+w)-a(y-x-z-w))\neq0.$$
So, for $n\geq3$, $M(AG(n,3),\frac{1}{4},\mathbb{F})$ is Jordan only if $\mathrm{ch}(\mathbb{F})=3$. 
\end{remark}

\section{Proof of Theorem}
It is proved that the Matsuo algebras associated with $FS(\mathit{Sym}(n),(1,2)^{\mathit{Sym}(n)})$ and $AG(2,3)$ are Jordan algebras if $\delta=\frac{1}{4}$ in \cite{dr15}.
By this result and the result in the preceding section, it is already proved that the Matsuo algebras associated with $FSS_n$ or $AG(n,3)$ are Jordan.
So it suffices to show that a Fischer space of rank $n$ is isomorphic to $FSS_{n+1}$ or $AG(n-1,3)$ if the associated Matsuo algebra is Jordan.

We proceed case-by-case for Fischer spaces of rank at most 4.
Fischer spaces of rank at most 3 are isomorphic to $FSS_2$, $FSS_3$, $FSS_4$ or $AG(2,3)$.

The Fischer spaces of rank 4 are classified in \cite{ch95}, Proposition 3.3.
They are associated with the six 3-transposition groups, including $\mathit{Sym}(5)$, M.\ Hall's $3^{10}:2$, and $\mathbb{F}_3^4\rtimes\mathbb{F}_3^{\times}$.
Let $(\mathcal{P},\mathcal{L})$ be a Fisher space associated with M.\ Hall's $3^{10}:2$ and $x,y,z,w$ be 3-transpositions generating $3^{10}:2$. Then,
\begin{eqnarray*}
&&J(a(x),a(y),a(z),a(w))\\
&&=\frac{1}{64}(a(x^{zwyw})+a(x^{wzyz})+a(x^{yzwz})-a(x^{zyw})-a(x^{wyz})-a(z^{wyx})).
\end{eqnarray*}
Since these six 3-transpositions of $3^{10}:2$ are distinct, the linearized Jordan identity does not hold.
So, the Matsuo algebra associated with $(\mathcal{P},\mathcal{L})$ is not Jordan.
By similar calculation for the other cases, the four Fischer spaces of rank 4 except $FSS_5$ and $AG(3,3)$ cannot give rise to Jordan algebras.

Suppose $n\geq5$.
Let $(\mathcal{P},\mathcal{L})$ be a Fischer space of rank $n$ with the associated Matsuo algebra Jordan.

Suppose $(\mathcal{P},\mathcal{L})$ has a subspace of $(\mathcal{P},\mathcal{L})$ associated with $\mathit{Sym}(5)$.
Let $(G,D)$ be a 3-transposition group such that $(\mathcal{P},\mathcal{L})$ is isomorphic to $FS(G,D)$.
To show that $G\cong\mathit{Sym}(n+1)$, let us construct a sequence $T_4\subseteq T_5\subseteq\cdots$ of subsets of $D$ such that the noncommuting graph $\mathcal{T}_r$ of $T_r$ is a line and $\langle T_r\rangle\cong\mathit{Sym}(|T_r|)$.
By the assumption, $G$ has a subgroup generated by 4-subset of $D$ and isomorphic to $\mathit{Sym}(5)$.
So there exists a $4$-subset of $D$ such that the noncommuting graph of it is a line.
Let $T_4$ be such a $4$-subset.
Assume that the set $T_r$ is a subset of $D$ satisfying the conditions above.
If $T_r$ generates $G$, then set $T_{r+1}=T_r$.
Suppose that $T_r$ does not generate $G$.
Then, for some $h\in D\setminus \langle T_r\rangle$, the noncommuting graph $\mathcal{T}_r(h)$ of $T_r\cup\{h\}$ is connected.
Let $h$ be such an element of $D$.
Since a rank $3$ quotient of $(H_3,E_3)$ cannot be a subgroup of $\mathit{Sym}(n)$, all subgroups generated by 4 elements of $T_r\cup \{h\}$ are isomorphic to $\mathit{Sym}(5)$.
In \cite{dr15}, it is proved that the noncommuting graph $\mathcal{T}$ of $T$ is not $\tilde{A}_{r-1}$ for each $T\subset D$ with $r(\langle T\rangle)\ge 4$.
Hence the number of elements of $T_r$ which do not commute with $h$ is at most two and these two elements do not commute. 
So there exists an element $h^{\prime}\in D$ such that $\langle T_r\cup \{h^{\prime}\}\rangle=\langle T_r\cup \{h\}\rangle$ and $\mathcal{T}_r(h^{\prime})$ is a line. 
Therefore $\langle T_r\cup\{h^{\prime}\}\rangle$ is isomorphic to $\mathit{Sym}(r+2)$.
Let $T_{r+1}=T_r\cup\{h^{\prime}\}$.
Then, since $|G|<\infty$, $G$ is generated by $T_r$ for some $r$ and hence $(G,D)$ is isomorphic to $(\mathit{Sym}(m),(1,2)^{\mathit{Sym}(m)})$ for some $m$.
Since $(\mathcal{P},\mathcal{L})$ is of rank $n$, $m=n+1$.

Suppose that all subspaces of $(\mathcal{P},\mathcal{L})$ of rank 4 are not isomorphic to $FSS_5$.
Then, all subspaces of rank 4 must be isomorphic to $AG(3,3)$.
So $(\mathcal{P},\mathcal{L})$ is of affine type and $\tau(x)\tau(y)\tau(z)=\tau(z)\tau(y)\tau(x)$ for all $x,y,z\in\mathcal{P}$. 
Then, by Corollary \ref{cor}, $(\mathcal{P},\mathcal{L})$ must be isomorphic to $AG(n-1,3)$.
The proof of  Theorem is completed.
\begin{remark}
When $\mathrm{ch}(\mathbb{F})\neq 3$, the Matsuo algebra associated with a Fischer space $(\mathcal{P},\mathcal{L})$ becomes Jordan if and only if $(\mathcal{P},\mathcal{L})$ satisfies the condition (i) or it is $AG(2,3)$. See \cite{dr15}.
\end{remark} 
\begin{remark}
Let $(G,D)$ be a connected 3-transposition group of rank $n$ such that the Matsuo algebra associated with $FS(G,D)$ is Jordan.
Then, by the results above, the central quotient of $(G,D)$ is isomorphic to $(\mathit{Sym}(n+1),(1,2)^{\mathit{Sym}(n+1)})$ or $(\mathbb{F}_3^{n-1}\rtimes\mathbb{F}_3^{\times},\{(v,-1)\mid v\in\mathbb{F}_3^{n-1}\})$.
\end{remark}
\begin{remark}
This theorem can be generalized to the infinite-dimensional case.
In this generalization, a Fischer space with the associated Matsuo algebras being Jordan is isomorphic to the Fischer space associated with the inductive limit of symmetric groups of finite degrees or that of finite-dimensional affine spaces.
\end{remark}

\bigbreak\noindent
Takahiro Yabe\\
Graduate School of Mathematical Sciences, The University of Tokyo\\
3-8-1, Komaba, Meguro-ku, Tokyo 153-8914, Japan
\\
Email: {\tt tyabe@ms.u-tokyo.ac.jp}
\end{document}